\documentclass[12pt]{article}
\usepackage{fullpage,graphicx,url}
\usepackage[small,bf]{caption}
\usepackage{amsmath,amssymb,amsthm,enumitem}
\usepackage{multirow}
\newcommand{\BEAS}{\begin{eqnarray*}}
\newcommand{\EEAS}{\end{eqnarray*}}
\newcommand{\BEQ}{\begin{equation}}
\newcommand{\EEQ}{\end{equation}}
\newcommand{\BIT}{\begin{itemize}}
\newcommand{\EIT}{\end{itemize}}

\newcommand{\eg}{{\it e.g.}}
\newcommand{\ie}{{\it i.e.}}

\newcommand{\ones}{\mathbf 1}
\newcommand{\reals}{{\mbox{\bf R}}}
\newcommand{\integers}{{\mbox{\bf Z}}}

\newcommand{\Tr}{\mathop{\bf Tr}}
\newcommand{\diag}{\mathop{\bf diag}}

\newcommand{\Expect}{\mathop{\bf E{}}}

\newcommand{\round}{\mathop{\bf round}}

\newcounter{oursection}

\newcounter{algorithmctr}[section]
\renewcommand{\thealgorithmctr}{\thesection.\arabic{algorithmctr}}
\newenvironment{algdesc}
   {\refstepcounter{algorithmctr}\begin{list}{}{
       \setlength{\rightmargin}{0\linewidth}
       \setlength{\leftmargin}{.05\linewidth}}
       \rmfamily\small
       \item[]{\setlength{\parskip}{0ex}\hrulefill\par
        \nopagebreak{\bfseries\textsf{Algorithm \thealgorithmctr~}}}}
   {{\setlength{\parskip}{-1ex}\nopagebreak\par\hrulefill} \end{list}}

\newtheorem{theorem}{Theorem}

\title{\Large{A Semidefinite Programming Method for Integer Convex Quadratic Minimization}}
\author{Jaehyun Park \and Stephen Boyd}

\begin{document}
\maketitle

\begin{abstract}
We consider the NP-hard problem of minimizing a convex quadratic function over
the integer lattice $\integers^n$. We present a simple semidefinite
programming (SDP) relaxation for obtaining a nontrivial lower bound on the
optimal value of the problem. By interpreting the solution to the SDP
relaxation probabilistically, we obtain a randomized algorithm for finding
good suboptimal solutions, and thus an upper bound on the optimal value. The
effectiveness of the method is shown for numerical problem instances of
various sizes.
\end{abstract}

\section{Introduction}

We consider the NP-hard problem
\BEQ\label{problem_statement}
\begin{array}{ll}
\mbox{minimize} & f(x) = x^T P x + 2q^T x \\
\mbox{subject to} & x \in \integers^n,
\end{array}
\EEQ
with variable $x$, where $P \in \reals^{n \times n}$ is nonzero, symmetric,
and positive semidefinite, and $q \in \reals^n$.

A number of other problems can be reduced to the form
of~(\ref{problem_statement}). The \emph{integer least squares problem},
\BEQ\label{ils_statement}
\begin{array}{ll}
\mbox{minimize} & \|Ax - b\|_2^2 \\
\mbox{subject to} & x \in \integers^n,
\end{array}
\EEQ
with variable $x$ and data $A \in \reals^{m \times n}$ and $b \in \reals^m$,
is easily reduced to the form of~(\ref{problem_statement}) by expanding out
the objective function. The mixed-integer
version of the problem, where some components of $x$ are allowed to be
real numbers, also reduces to an equivalent problem with integer variables
only. This transformation uses the Schur complement to explicitly minimize over
the noninteger variables~\cite[\S A.5.5]{boyd2004convex}. Another equivalent
formulation of~(\ref{problem_statement}) is the \emph{closest vector problem},
\[
\begin{array}{ll}
\mbox{minimize} & \|v - z\|_2^2 \\
\mbox{subject to} & z \in \{Bx \,|\, x \in \integers^n \},
\end{array}
\]
in the variable $z \in \reals^m$. Typically, the columns of $B$ are linearly
independent. Although not equivalent to~(\ref{problem_statement}), the
\emph{shortest vector problem} is also a closely related problem, which in
fact, is reducible to solving the closest vector problem:
\[
\begin{array}{ll}
\mbox{minimize} & \|z\|_2^2 \\
\mbox{subject to} & z \in \{Bx \,|\, x \in \integers^n \} \\
& z \ne 0.
\end{array}
\]

Problem~(\ref{problem_statement}) arises in several applications. For example,
in position estimation using the Global Positioning System (GPS), resolving
the integer ambiguities of the phase data is posed as a mixed-integer least
squares problem~\cite{hassibi1998integer}. In multiple-input multiple-output
(MIMO) wireless communication systems, maximum likelihood detection of
(vector) Boolean messages involves solving an integer least squares
problem~\cite{jalden2005complexity}. The mixed-integer version of the least
squares problem appears in data fitting applications, where some parameters
are integer-valued. (See, \eg,~\cite{ustun2016supersparse}.) The closest
vector problem and shortest vector problem have numerous application areas in
cryptanalysis of public key cryptosystem such as RSA~\cite{nguyen2001two}. The
\emph{spectral test}, which is used to check the quality of linear
congruential random number generators, is an application of the shortest
vector problem~\cite[\S3.3.4]{knuth1997taocp2}.

\subsection{Previous work}

Several hardness results are known for the integer least squares
problem~(\ref{ils_statement}). Given an instance of the integer least squares
problem, define the approximation factor of a point $x$ to be
$\|Ax-b\|_2^2 / \|Ax^\star - b\|_2^2$, where $x^\star$ is the global (integer)
solution of~(\ref{ils_statement}). Finding a constant factor approximation is
an NP-hard problem~\cite{arora1993hardness}. In fact, finding an
approximation still remains NP-hard even when the target approximation factor
is relaxed to $n^{c/\log \log n}$, where $c>0$ is some
constant~\cite{dinur2003approximating}.

Standard methods for finding the global optimum of~(\ref{problem_statement}),
in the case of positive definite $P$, work by enumerating all integer points
within a suitably chosen box or
ellipsoid~\cite{fincke1985improved,buchheim2012effective}. The worst case
running time of these methods is exponential in $n$, making it impractical for
problems of large size. Algorithms such as Lenstra--Lenstra--Lov{\'a}sz
lattice reduction algorithm~\cite{lenstra1982factoring,schnorr1994lattice} can
be used to find an approximate solution in polynomial time, but the
approximation factor guarantee is exponential in
$n$~\cite[\S5.3]{grotschel2012geometric}.

A simple lower bound on $f^\star$, the optimal value
of~(\ref{problem_statement}), can be obtained in $O(n^3)$ time, by removing
the integer constraint. If $q \in \mathcal{R}(P)$, where $\mathcal{R}(P)$
denotes the range of $P$, then this continuous relaxation has a solution
$x^\mathrm{cts} = -P^\dagger q$, with objective value $f^\mathrm{cts} = -q^T
P^\dagger q$, where $P^\dagger$ denotes the Moore--Penrose pseudoinverse of
$P$. (When $P$ is positive definite, the continuous solution reduces to
$x^\mathrm{cts} = -P^{-1} q$.) If $q \notin \mathcal{R}(P)$, then the
objective function is unbounded below and $f^\star = -\infty$.

There exist different approaches for obtaining tighter lower bounds than
$f^\mathrm{cts}$. The strongest bounds to date are based on semidefinite
programming (SDP)
relaxation~\cite{buchheim2012effective,buchheim2013semidefinite,buchheim2015ellipsoid}.
The primary drawback of the SDP-based methods is
their running time. In particular, if these methods are applied to
branch-and-bound type enumeration methods to prune the search tree, the
benefit of having a stronger lower bound is overshadowed by the additional
computational cost it incurs, for all small- to medium-sized problems.
Enumeration methods still take exponential time in the number of variables,
whereas solving SDPs can be done (practically) in $O(n^3)$ time. Thus, for
very large problems, SDP-based lower bounds are expected to reduce the total
running time of the enumeration methods. However, such problems would be too
big to have any practical implication. On the other hand, there exist weaker
bounds that are quicker to compute; in~\cite{buchheim2015ellipsoid}, for
example, these bounds are obtained by finding a quadratic function $\tilde f$
that is a global underestimator of $f$, that has the additional property that
the integer point minimizing $\tilde f$ can be found simply by rounding
$x^\mathrm{cts}$ to the nearest integer point. Another approach is given
by~\cite{bienstock2010eigenvalue}, which is to minimize $f$ outside an
ellipsoid that can be shown to contain no integer point. Standard results on
the $\mathcal{S}$-procedure state that optimizing a quadratic function outside
an ellipsoid, despite being a nonconvex problem, can be done exactly and
efficiently~\cite{boyd1994linear}.

A simple upper bound on $f^\star$ can be obtained by observing some properties
of the problem. First of all, $x=0$ gives a trivial upper bound of $f(0) = 0$,
which immediately gives $f^\star \le 0$. Another simple approximate solution
can be obtained by rounding each entry of $x^\mathrm{cts}$ to the nearest
integer point, $x^\mathrm{rnd}$. Let $f^\mathrm{rnd} = f(x^\mathrm{rnd})$.
Assuming that $q \in \mathcal{R}(P)$, we can get a bound on $f^\mathrm{rnd}$
as follows. Start by rewriting the objective function as
\[
f(x) = (x-x^\mathrm{cts})^T P (x-x^\mathrm{cts}) + f^\mathrm{cts}.
\]
Since rounding changes each coordinate by at most $1/2$, we have
\[
\|x^\mathrm{rnd} - x^\mathrm{cts}\|_2^2 =
\sum_{i=1}^n (x^\mathrm{rnd}_i - x^\mathrm{cts}_i)^2 \le n/4.
\]
It follows that
\BEQ\label{ubround}
f^\mathrm{rnd} - f^\mathrm{cts} =
(x^\mathrm{rnd}-x^\mathrm{cts})^T P (x^\mathrm{rnd}-x^\mathrm{cts}) \le
\sup_{\|v\|_2 \le \sqrt{n}/2} v^T P v = (n/4)\omega_{\max},
\EEQ
where $\omega_{\max}$ is the largest eigenvalue of $P$. Since $f^\mathrm{cts}$
is a lower bound on $f^\star$, this inequality bounds the suboptimality of
$x^\mathrm{rnd}$. We note that in the special case of diagonal $P$, the
objective function is separable, and thus the rounded solution is optimal.
However, in general, $x^\mathrm{rnd}$ is not optimal, and in fact,
$f^\mathrm{rnd}$ can be positive, which is even worse than the trivial upper
bound $f(0) = 0$.

We are not aware of any efficient method of finding a strong upper bound on
$f^\star$, other than performing a local search or similar heuristics on
$x^\mathrm{rnd}$. However, the well-known result by~\cite{goemans1995improved}
gives provable lower and upper bounds on the optimal value of the NP-hard
\emph{maximum cut problem}, which, after a simple reformulation, can be cast
as a Boolean nonconvex quadratic problem in the following form:
\BEQ\label{maxcut}
\begin{array}{ll}
\mbox{maximize} & x^T W x \\
\mbox{subject to} & x_i^2 = 1, \quad i=1, \ldots, n.
\end{array}
\EEQ
These bounds were obtained by solving an SDP relaxation of~(\ref{maxcut}), and
subsequently running a randomized algorithm using the solution of the
relaxation. The expected approximation factor of the randomized algorithm is
approximately $0.878$. There exist many extensions of the Goemans--Williamson
SDP relaxation~\cite{luo2010semidefinite}. In
particular,~\cite{buchheim2013semidefinite} generalizes this idea to a more
general domain $D_1 \times \cdots \times D_n$, where each $D_i$ is any closed
subset of $\reals$.

\subsection{Our contribution}
Our aim is to present a simple but powerful method of producing \emph{both}
lower and upper bounds on the optimal value $f^\star$
of~(\ref{problem_statement}). Our SDP relaxation is an adaptation
of~\cite{goemans1995improved}, but can also be recovered by appropriately
using the method in~\cite{buchheim2013semidefinite}.
By considering the binary expansion of the
integer variables as a Boolean variable, we can
reformulate~(\ref{problem_statement}) as a Boolean problem
and directly apply the method of~\cite{goemans1995improved}.
This reformulation, however, increases the size of the problem
and incurs additional computational cost.
To avoid this, we work with the
formulation~(\ref{problem_statement}), at the expense of slightly looser SDP-based bound.
We show that our lower bound still consistently outperforms other lower
bounds shown in~\cite{buchheim2012effective,buchheim2015ellipsoid}.
In particular, the new bound is better than the best axis-parallel ellipsoidal
bound, which also requires solving an SDP.

Using the solution of the SDP relaxation, we construct a randomized algorithm that finds good
feasible points. In addition, we present a simple local search heuristic that
can be applied to every point generated by the randomized algorithm.
Evaluating the objective at these points gives an upper bound on the optimal value.
This upper bound provides a
good starting point for enumeration methods, and can save a significant amount
of time during the search process. We show this by comparing the running time
of an enumeration method, when different initial upper bounds on the optimal
value were given. Also, we empirically verify that this upper bound is much
stronger than simply rounding a fractional solution to the nearest integer
point, and in fact, is near-optimal for randomly generated problem instances.

\section{Lagrange duality} \label{s-lagrange}
In this section, we discuss a Lagrangian relaxation for obtaining a nontrivial
lower bound on $f^\star$. We make three assumptions without
loss of generality. Firstly, we assume that $q \in \mathcal{R}(P)$, so that
the optimal value $f^\star$ is not unbounded below. Secondly, we assume that
$x^\mathrm{cts} \notin \integers^n$, otherwise $x^\mathrm{cts}$ is already
the global solution. Lastly, we assume that $x^\mathrm{cts}$ is in the box
$[0, 1]^n$. For any arbitrary problem instance, we can translate the
coordinates in the following way to satisfy this assumption. Note that for any
$v \in \integers^n$, the problem below is equivalent
to~(\ref{problem_statement}):
\[
\begin{array}{ll}
\mbox{minimize} & (x-v)^T P (x-v) + 2(Pv+q)^T (x-v) + f(v) \\
\mbox{subject to} & x \in \integers^n.
\end{array}
\]
By renaming $x-v$ to $x$ and ignoring the constant term $f(v)$, the problem
can be rewritten in the form of~(\ref{problem_statement}). Clearly, this has
different solutions and optimal value from the
original problem, but the two problems are related by a simple change of
coordinates: point $x$ in the new problem corresponds to $x+v$ in the
original problem. To translate the coordinates, find
$x^\mathrm{cts} = -P^\dagger q$, and take elementwise floor to
$x^\mathrm{cts}$ to get $x^\mathrm{flr}$. Then, substitute $x^\mathrm{flr}$ in
place of $v$ above.

We note a simple fact that every integer point $x$ satisfies either $x_i \le
0$ or $x_i \ge 1$ for all $i$. Equivalently, this condition can be written as
$x_i(x_i - 1) \ge 0$ for all $i$. Using this, we relax the integer constraint
$x \in \integers^n$ into a set of nonconvex quadratic constraints:
$x_i(x_i - 1) \ge 0$ for all $i$.
The following nonconvex problem is then
a relaxation of~(\ref{problem_statement}):
\BEQ\label{nonconvex_relaxation}
\begin{array}{ll}
\mbox{minimize} & x^T P x + 2 q^T x \\
\mbox{subject to} & x_i(x_i-1) \ge 0, \quad i=1, \ldots, n.
\end{array}
\EEQ
It is easy to see that the optimal value of~(\ref{nonconvex_relaxation}) is
greater than or equal to $f^\mathrm{cts}$, because $x^\mathrm{cts}$ is not a
feasible point, due to the two assumptions that $x^\mathrm{cts} \notin
\integers^n$ and $x^\mathrm{cts} \in [0, 1]^n$. Note that the second
assumption was necessary, for otherwise $x^\mathrm{cts}$ is the global optimum
of~(\ref{nonconvex_relaxation}), and the Lagrangian relaxation described below
would not produce a lower bound that is better than $f^\mathrm{cts}$.

The Lagrangian of~(\ref{nonconvex_relaxation}) is given by
\[
L(x, \lambda) = x^T P x + 2q^T x - \sum_{i=1}^n \lambda_i x_i(x_i-1) = x^T (P - \diag(\lambda)) x + 2(q + (1/2)\lambda)^T x,
\]
where $\lambda \in \reals^n$ is the vector of dual variables.
Define $\tilde q(\lambda) = q + (1/2)\lambda$. By minimizing the Lagrangian
over $x$, we get the Lagrangian dual function
\BEQ\label{dual_function}
g(\lambda) = \left\{ \begin{array}{ll}
-\tilde{q}(\lambda)^T \left(P - \diag(\lambda) \right)^\dagger \tilde q(\lambda) \quad
& \mbox{if } P-\diag(\lambda) \succeq 0 \mbox{ and } \tilde q(\lambda) \in \mathcal{R}(P-\diag(\lambda)) \\
-\infty & \mbox{otherwise,} \end{array} \right.
\EEQ
where the inequality $\succeq$ is with respect to the positive semidefinite
cone. The Lagrangian dual problem is then
\BEQ\label{dual_relaxation}
\begin{array}{ll}
\mbox{maximize} & g(\lambda) \\
\mbox{subject to} & \lambda \ge 0,
\end{array}
\EEQ
in the variable $\lambda \in \reals^n$, or equivalently,
\[
\begin{array}{ll}
\mbox{maximize} & -\tilde{q}(\lambda)^T \left(P - \diag(\lambda) \right)^\dagger \tilde q(\lambda) \\
\mbox{subject to} & P-\diag(\lambda) \succeq 0 \\
& \tilde q(\lambda) \in \mathcal{R}(P-\diag(\lambda)) \\
& \lambda \ge 0.
\end{array}
\]
By using the Schur complements, the problem can be reformulated into an SDP:
\BEQ\label{dual_sdp_form}
\begin{array}{ll}
\mbox{maximize} & -\gamma \\
\mbox{subject to} & \left[ \begin{array}{cc}
P - \diag(\lambda) & q + (1/2)\lambda \\
(q + (1/2)\lambda)^T & \gamma \end{array} \right] \succeq 0 \\
& \lambda \ge 0,
\end{array}
\EEQ
in the variables $\lambda \in \reals^n$ and $\gamma \in \reals$.
We note that while~(\ref{dual_sdp_form}) is derived from a nonconvex
problem~(\ref{nonconvex_relaxation}), it is convex and thus can be
solved in polynomial time.

\subsection{Comparison to simple lower bound}

Due to weak duality, we have $g(\lambda) \le f^\star$ for any
$\lambda \ge 0$, where $g(\lambda)$ is defined by~(\ref{dual_function}). Using
this property, we show a provable bound on the Lagrangian lower bound. Let
$f^\mathrm{cts} = -q^T P^\dagger q$ be the
simple lower bound on $f^\star$, and
$f^\mathrm{sdp} = \sup_{\lambda \ge 0} g(\lambda)$ be the lower bound obtained
by solving the Lagrangian dual. Also, let $\omega_1 \ge \cdots \ge \omega_n$
be the eigenvalues of $P$. For clarity of notation, we use $\omega_{\max}$ and
$\omega_{\min}$ to denote the largest and smallest eigenvalues of $P$, namely
$\omega_1$ and $\omega_n$. Let $\ones$ represent a vector of an appropriate
length with all components equal to one. Then, we have the following result.

\begin{theorem}\label{t-lbguarantee}
The lower bounds satisfy
\BEQ\label{lbguarantee}
f^\mathrm{sdp} - f^\mathrm{cts} \ge \frac{n \omega_{\min}^2}{4\omega_{\max}} \left( 1 - \frac{\left\|x^\mathrm{cts} - (1/2)\ones \right\|_2^2}{n/4}\right)^2.
\EEQ
\end{theorem}

\begin{proof}
When $\omega_{\min} = 0$, the righthand side of~(\ref{lbguarantee}) is zero,
and there is nothing else to show. Thus, without loss of generality, we assume
that $\omega_{\min} > 0$, \ie, $P \succ 0$.

Let $P = Q \diag(\omega) Q^T$ be the eigenvalue decomposition of $P$, where
$\omega = (\omega_1, \ldots, \omega_n)$. We consider $\lambda$ of the form
$\lambda = \alpha \ones$, and rewrite the dual function in terms of $\alpha$,
where $\alpha$ is restricted to the range $\alpha \in [0, \omega_{\min})$:
\[
g(\alpha) = -(q + (1/2)\alpha\ones)^T \left(P - \alpha I \right)^{-1} (q + (1/2)\alpha\ones).
\]
We note that $g(0) = -q^T P^{-1} q = f^\mathrm{cts}$, so it is enough to show
the same lower bound on $g(\alpha) - g(0)$ for any particular value of
$\alpha$.

Let $s = Q^T \ones$, and $\tilde q = Q^T q$. By expanding out $g(\alpha)$ in
terms of $s$, $\tilde q$, and $\omega$, we get
\BEAS
g(\alpha) - g(0) &=& -\sum_{i=1}^n \frac{\tilde q_i^2 + \alpha s_i \tilde q_i + (1/4)\alpha^2 s_i^2}{\omega_i - \alpha} + \sum_{i=1}^n \frac{\tilde q_i^2}{\omega_i}\\
&=& -\sum_{i=1}^n \frac{(\alpha/\omega_i) (\tilde q_i+(1/2)\omega_i s_i)^2 - (1/4)\alpha s_i^2(\omega_i - \alpha)}{\omega_i - \alpha} \\
&=& \frac{\alpha}{4} \sum_{i=1}^n s_i^2 - \sum_{i=1}^n \frac{\alpha (\tilde q_i+(1/2)\omega_i s_i)^2}{\omega_i(\omega_i - \alpha)} \\
&=& \frac{\alpha n}{4} - \alpha \sum_{i=1}^n \left(1 - \frac{\alpha}{\omega_i}\right) \left( \frac{\tilde q_i+(1/2)\omega_i s_i}{\omega_i - \alpha}\right)^2.
\EEAS
By differentiating the expression above with respect to $\alpha$, we get
\[
g'(\alpha) = \frac{n}{4} - \sum_{i=1}^n \left( \frac{\tilde q_i+(1/2)\omega_i s_i}{\omega_i - \alpha}\right)^2.
\]
We note that $g'$ is a decreasing function in $\alpha$ in the interval
$[0, \omega_{\min})$. Also, at $\alpha = 0$, we have
\BEAS
g'(0) &=& \frac{n}{4} - \sum_{i=1}^n ( \tilde q_i/\omega_i +(1/2) s_i)^2 \\
&=& \frac{n}{4} - \left\|\diag(\omega)^{-1} Q^T q + (1/2) Q^T \ones\right\|_2^2 \\
&=& \frac{n}{4} - \left\|-Q \diag(\omega)^{-1} Q^T q - (1/2) QQ^T \ones\right\|_2^2 \\
&=& \frac{n}{4} - \left\|-P^{-1} q - (1/2) \ones\right\|_2^2 \\
&=& \frac{n}{4} - \left\|x^\mathrm{cts} - (1/2) \ones\right\|_2^2 \\
&\ge& 0.
\EEAS
The last line used the fact that $x^\mathrm{cts}$ is in the box $[0, 1]^n$.

Now, we distinguish two cases depending on whether the equation $g'(\alpha) = 0$
has a solution in the interval $[0, \omega_{\min})$.
\begin{enumerate}
\item Suppose that $g'(\alpha^\star) = 0$ for
some $\alpha^\star \in [0, \omega_{\min})$. Then, we have
\BEAS
g(\alpha^\star) - g(0) &=& \frac{\alpha^\star n}{4} - \alpha^\star \sum_{i=1}^n \left(1 - \frac{\alpha^\star}{\omega_i}\right) \left( \frac{\tilde q_i+(1/2)\omega_i s_i}{\omega_i - \alpha^\star}\right)^2 \\
&=& \alpha^\star \left(\frac{n}{4} - \sum_{i=1}^n \left( \frac{\tilde q_i+(1/2)\omega_i s_i}{\omega_i - \alpha^\star}\right)^2 \right) + \sum_{i=1}^n \frac{\alpha^{\star 2}}{\omega_i} \left( \frac{\tilde q_i+(1/2)\omega_i s_i}{\omega_i - \alpha^\star}\right)^2 \\
&=& \sum_{i=1}^n \frac{\alpha^{\star 2}}{\omega_i}\left(\frac{\tilde q_i + (1/2)\omega_i s_i}{\omega_i - \alpha^\star}\right)^2 \\
&\ge& \frac{\alpha^{\star 2}}{\omega_{\max}} \sum_{i=1}^n \left(\frac{\tilde q_i + (1/2)\omega_i s_i}{\omega_i - \alpha^\star}\right)^2 \\
&=& \frac{n \alpha^{\star 2}}{4 \omega_{\max}}.
\EEAS
Using this, we go back to the equation $g'(\alpha^\star) = 0$ and
establish a lower bound on $\alpha^\star$:
\BEAS
\frac{n}{4} &=& \sum_{i=1}^n \left( \frac{\tilde q_i+(1/2)\omega_i s_i}{\omega_i - \alpha^\star}\right)^2 \\
&=& \sum_{i=1}^n \frac{\omega_i}{\omega_i - \alpha^\star} ( \tilde q_i / \omega_i + (1/2) s_i)^2 \\
&\le& \frac{\omega_{\min}}{\omega_{\min} - \alpha^\star} \sum_{i=1}^n ( \tilde q_i / \omega_i + (1/2) s_i)^2 \\
&=& \frac{\omega_{\min}}{\omega_{\min} - \alpha^\star} \left\|x^\mathrm{cts} - (1/2) \ones\right\|_2^2.
\EEAS
From this inequality, we have
\[
\alpha^\star \ge \omega_{\min}\left(1-\frac{\|x^\mathrm{cts}-(1/2)\ones\|_2^2}{n/4}\right).
\]
Plugging in this lower bound on $\alpha^\star$ gives
\[
f^\mathrm{sdp} - g(0) \ge g(\alpha^\star) - g(0) \ge \frac{n\omega_{\min}^2}{4 \omega_{\max}}\left( 1 - \frac{\|x^\mathrm{cts}-(1/2)\ones\|_2^2}{n/4} \right)^2.
\]

\item If $g'(\alpha) \ne 0$ for all $\alpha \in [0, \omega_{\min})$, then by
continuity of $g'$, it must be the case that $g'(\alpha) \ge 0$ on the
range $[0, \omega_{\min})$.
Then, for all $\alpha \in [0, \omega_{\min})$,
\BEAS
f^\mathrm{sdp} - g(0) &\ge& g(\alpha)-g(0) \\
&=& \frac{\alpha n}{4} - \alpha \sum_{i=1}^n \left(1 - \frac{\alpha}{\omega_i}\right) \left( \frac{\tilde q_i+(1/2)\omega_i s_i}{\omega_i - \alpha}\right)^2 \\
&\ge& \frac{\alpha n}{4} - \alpha \left(1 - \frac{\alpha}{\omega_{\max}}\right) \sum_{i=1}^n \left( \frac{\tilde q_i+(1/2)\omega_i s_i}{\omega_i - \alpha}\right)^2 \\
&\ge& \frac{\alpha n}{4} - \alpha \left(1 - \frac{\alpha}{\omega_{\max}}\right) \frac{n}{4} \\
&=& \frac{n \alpha^2}{4\omega_{\max}},
\EEAS
and thus,
\[
f^\mathrm{sdp} - g(0) \ge \lim_{\alpha \rightarrow \omega_{\min}} \frac{n \alpha^2}{4\omega_{\max}} = \frac{n \omega_{\min}^2}{4 \omega_{\max}}.
\]
\end{enumerate}

Therefore, in both cases, the increase in the lower bound is guaranteed
to be at least
\[
\frac{n\omega_{\min}^2}{4 \omega_{\max}}\left( 1 - \frac{\|x^\mathrm{cts}-(1/2)\ones\|_2^2}{n/4} \right)^2,
\]
as claimed.
\end{proof}

Now we discuss several implications of Theorem~\ref{t-lbguarantee}. First, we
note that the righthand side of~(\ref{lbguarantee}) is always nonnegative, and
is monotonically decreasing in
$\left\|x^\mathrm{cts} - (1/2)\ones \right\|_2$. In particular, when
$x^\mathrm{cts}$ is an integer point, then we must have
$f^\mathrm{cts} = f^\mathrm{sdp} = f^\star$.
Indeed, for $x^\mathrm{cts} \in \{0, 1\}^n$, we
have $\left\|x^\mathrm{cts} - (1/2)\ones \right\|_2^2 = n/4$, and the
righthand side of~(\ref{lbguarantee}) is zero. Also, when $P$ is positive
definite, \ie, $\omega_{\min} > 0$, then~(\ref{lbguarantee}) implies that
$f^\mathrm{sdp} > f^\mathrm{cts}$.

In order to obtain the bound on $f^\mathrm{sdp}$, we only considered vectors
$\lambda$ of the form $\alpha \ones$. Interestingly,
solving~(\ref{dual_relaxation}) with this additional restriction is equivalent
to solving the following problem:
\BEQ\label{sproc_relaxation}
\begin{array}{ll}
\mbox{minimize} & x^T P x + 2q^T x \\
\mbox{subject to} & \|x - (1/2)\ones\|_2^2 \ge n/4.
\end{array}
\EEQ
The nonconvex constraint enforces that $x$ lies outside the $n$-dimensional
sphere centered at $(1/2)\ones$ that has every lattice point $\{0, 1\}^n$ on
its boundary. Even if~(\ref{sproc_relaxation}) is not a convex problem, it can
be solved exactly; the $\mathcal{S}$-lemma implies that the SDP relaxation
of~(\ref{sproc_relaxation}) is tight (see, \eg,~\cite{boyd1994linear}). For
completeness, we give the dual of the SDP relaxation (which has the same
optimal value as~(\ref{sproc_relaxation})) below, which is exactly what we
used to prove Theorem~\ref{t-lbguarantee}:
\[
\begin{array}{ll}
\mbox{maximize} & -\gamma \\
\mbox{subject to} & \left[ \begin{array}{cc}
P - \alpha I & q + (\alpha/2)\ones \\
(q + (\alpha/2)\ones)^T & \gamma \end{array} \right] \succeq 0 \\
& \alpha \ge 0.
\end{array}
\]

\section{Semidefinite relaxation}
In this section, we show another convex relaxation
of~(\ref{nonconvex_relaxation}) that is equivalent to~(\ref{dual_sdp_form}),
but with a different form. By introducing a new variable $X = x x^T$, we can
reformulate~(\ref{nonconvex_relaxation}) as:
\[
\begin{array}{ll}
\mbox{minimize} & \Tr(P X) + 2q^T x \\
\mbox{subject to} & \diag(X) \ge x \\
& X = x x^T,
\end{array}
\]
in the variables $X \in \reals^{n \times n}$ and $x \in \reals^n$. Observe
that the constraint $\diag(X) \ge x$, along with $X = xx^T$, is a rewriting
of the constraint $x_i(x_i - 1) \ge 0$ in~(\ref{nonconvex_relaxation}).

Then, we relax the nonconvex constraint $X = x x^T$ into $X \succeq x x^T$,
and rewrite it using the Schur complement to obtain a convex
relaxation:
\BEQ\label{primal_relaxation}
\begin{array}{ll}
\mbox{minimize} & \Tr(P X) + 2q^T x \\
\mbox{subject to} & \diag(X) \ge x \\
& \left[ \begin{array}{cc} X & x \\ x^T & 1 \end{array} \right] \succeq 0.
\end{array}
\EEQ
The optimal value of problem~(\ref{primal_relaxation}) is a lower bound on
$f^\star$, just as the Lagrangian relaxation~(\ref{dual_sdp_form}) gives a
lower bound $f^\mathrm{sdp}$ on $f^\star$. In fact,
problems~(\ref{dual_sdp_form}) and~(\ref{primal_relaxation}) are duals of each
other, and they yield the same lower bound
$f^\mathrm{sdp}$~\cite{vandenberghe1996semidefinite}.

\subsection{Randomized algorithm} \label{s-rand}

The semidefinite relaxation~(\ref{primal_relaxation}) has a natural
probabilistic interpretation, which can be used to construct a simple
randomized algorithm for obtaining good suboptimal solutions, \ie, feasible
points with low objective value. Let $(X^\star, x^\star)$ be any
solution to~(\ref{primal_relaxation}). Suppose $z \in \reals^n$ is a Gaussian
random variable with mean $\mu$ and covariance matrix $\Sigma$. Then, $\mu =
x^\star$ and $\Sigma = X^\star - x^\star x^{\star T}$ solve the following
problem of minimizing the expected value of a quadratic form, subject to
quadratic inequalities:
\[
\begin{array}{ll}
\mbox{minimize} & \Expect(z^T P z + 2q^T z) \\
\mbox{subject to} & \Expect (z_i(z_i-1)) \ge 0, \quad i=1, \ldots, n,
\end{array}
\]
in variables $\mu \in \reals^n$ and $\Sigma \in \reals^{n\times n}$.
Intuitively, this distribution $\mathcal{N}(\mu, \Sigma)$ has mean close to
$x^\mathrm{cts}$ so that the expected objective value is low, but each
diagonal entry of $\Sigma$ is large enough so that when $z$ is sampled from
the distribution, $z_i(z_i-1) \ge 0$ holds in expectation. While sampling $z$
from $\mathcal{N}(\mu, \Sigma)$ does not give a feasible point
to~(\ref{problem_statement}) immediately, we can simply round it to the
nearest integer point to get a feasible point. Using these observations, we
present the following randomized algorithm.

\begin{algdesc}\label{randalg}
\emph{Randomized algorithm for suboptimal solution to~(\ref{problem_statement}).}
\begin{tabbing}
{\bf given} number of iterations $K$. \\*[\smallskipamount]
1.\ \emph{Solve SDP.} Solve~(\ref{primal_relaxation}) to get $X^\star$ and $x^\star$.\\
2.\ \emph{Form covariance matrix.} $\Sigma := X^\star - x^\star x^{\star T}$, and find Cholesky factorization $LL^T = \Sigma$.\\
3.\ \emph{Initialize best point.} $x^\mathrm{best} := 0$ and $f^\mathrm{best} := 0$.\\
{\bf for} $k=1, 2, \ldots, K$ \\
\qquad \= 4.\ \emph{Random sampling.} $z^{(k)} := x^\star + Lw$, where $w \sim \mathcal{N}(0, I)$. (Same as $z^{(k)} \sim \mathcal{N}(x^\star, \Sigma)$.) \\
\> 5.\ \emph{Round to nearest integer.} $x^{(k)} := \round(z^{(k)})$. \\
\> 6.\ \emph{Update best point.} If $f^\mathrm{best} > f(x^{(k)})$, then set $x^\mathrm{best} := x^{(k)}$ and $f^\mathrm{best} := f(x^{(k)})$ .
\end{tabbing}
\end{algdesc}

The SDP in step 1 takes $O(n^3)$ time to solve, assuming that the number of
iterations required by an interior point method is constant. Step 2 is
dominated by the computation of Cholesky factorization, which uses roughly
$n^3/3$ flops. Steps 4 through 6 can be done in $O(n^2)$ time. The overall
time complexity of the method is then $O(n^2(K+n))$. By choosing $K = O(n)$,
the time complexity can be made $O(n^3)$.

\section{Greedy algorithm for obtaining a 1-opt solution} \label{s-oneopt}

Here we discuss a simple greedy descent algorithm that starts from an integer
point, and iteratively moves to another integer point that has a lower
objective value. This method can be applied to the simple suboptimal point
$x^\mathrm{rnd}$, or every $x^{(k)}$ found in Algorithm~\ref{randalg}, to
yield better suboptimal points.

We say that $x \in \integers^n$ is \emph{1-opt} if the objective value at $x$
does not improve by changing a single coordinate, \ie, $f(x + c e_i) \ge f(x)$
for all indices $i$ and integers $c$. The difference in the function values at
$x$ and $x+ce_i$ can be written as
\[
f(x+c e_i) - f(x) = c^2 P_{ii}+cg_i = P_{ii}(c+g_i/(2P_{ii}))^2 - g_i^2/(4P_{ii}),
\]
where $g = 2(Px + q)$ is the gradient of $f$ at $x$. It is easily seen that
given $i$, the expression above is minimized when
$c = \round(-g_i/(2P_{ii}))$.
For $x$ to be optimal with respect to $x_i$, then $c$
must be $0$, which is the case if and only if
$P_{ii} \ge |g_i|$. Thus, $x$ is 1-opt if and only if
$\diag(P) \ge |g|$, where the absolute value on the righthand side is taken
elementwise.

Also, observe that
\[
P(x + c e_i)+q = (Px + q) + c P_i,
\]
where $P_i$ is the $i$th column of $P$. Thus, when $x$ changes by a single
coordinate, the value of $g$ can be updated just by referencing a single
column of $P$. These observations suggest a simple and quick greedy algorithm
for finding a 1-opt point from any given integer point $x$.

\begin{algdesc}\label{oneoptalg}
\emph{Greedy descent algorithm for obtaining 1-opt point.}
\begin{tabbing}
{\bf given} an initial point $x \in \integers^n$. \\*[\smallskipamount]
1.\ \emph{Compute initial gradient.} $g = 2(Px + q)$.\\
{\bf repeat} \\
\qquad \= 2.\ \emph{Stopping criterion.} {\bf quit} if $\diag(P) \ge |g|$. \\
\> 3.\ \emph{Find descent direction.} Find index $i$ and integer $c$ minimizing $c^2 P_{ii} + c g_i$. \\
\> 4.\ \emph{Update $x$.} $x_i := x_i + c$. \\
\> 5.\ \emph{Update gradient.} $g := g + 2c P_i$.
\end{tabbing}
\end{algdesc}

Initializing $g$ takes $O(n^2)$ flops, but each subsequent iteration only
takes $O(n)$ flops. This is because steps 2 and 3 only refer to the diagonal
elements of $P$, and step 5 only uses a single column of $P$. Though we do not
give an upper bound on the total number of iterations, we show, using
numerical examples, that the average number of iterations until convergence is
roughly $0.14n$, when the initial points are sampled according to the
probability distribution given in~\S\ref{s-rand}. The overall time complexity
of Algorithm~\ref{oneoptalg} is then $O(n^2)$ on average. Thus, we can run the
greedy 1-opt descent on every $x^{(k)}$ in Algorithm~\ref{randalg}, without
changing its overall time complexity $O(n^2(K+n))$.

\section{Examples}

In this section, we consider numerical examples to show the performance of the
SDP-based lower bound and randomized algorithm, developed in previous
sections.

\subsection{Method}

We combine the techniques developed in previous sections to find lower and
upper bounds on $f^\star$, as well as suboptimal solutions to the problem. By solving the
simple relaxation and rounding the solution, we immediately get a lower bound
$f^\mathrm{cts}$ and an upper bound $f^\mathrm{rnd} = f(x^\mathrm{rnd})$. We
also run Algorithm~\ref{oneoptalg} on $x^\mathrm{rnd}$ to get a 1-opt point,
namely ${\hat x}^\mathrm{rnd}$. This gives another upper bound ${\hat
f}^\mathrm{rnd} = f({\hat x}^\mathrm{rnd})$.

Then, we solve the semidefinite relaxation~(\ref{primal_relaxation}) to get a
lower bound $f^\mathrm{sdp}$. Using the solution to the SDP, we run
Algorithm~\ref{randalg} to obtain suboptimal solutions, and keep the best
suboptimal solution $x^\mathrm{best}$. In addition, we run
Algorithm~\ref{oneoptalg} on every feasible point considered in step 4 of
Algorithm~\ref{randalg}, and find the best 1-opt suboptimal solution ${\hat
x}^\mathrm{best}$. The randomized algorithm thus yields two additional upper
bounds on $f^\star$, namely $f^\mathrm{best} =
f(x^\mathrm{best})$ and ${\hat f}^\mathrm{best} = f({\hat x}^\mathrm{best})$.

The total number of iterations $K$ in Algorithm~\ref{randalg} is set to $K =
3n$, so that the overall time complexity of the algorithm, not counting the
running time of the 1-opt greedy descent algorithm, is $O(n^3)$. We note that
the process of sampling points and running Algorithm~\ref{oneoptalg} trivially
parallelizes.

\subsection{Numerical examples}\label{s-prob-generation}

We use random instances of the integer least squares
problem~(\ref{ils_statement}) generated in the following way. First, the
entries of $A \in \reals^{m\times n}$ are sampled independently from
$\mathcal{N}(0, 1)$. The dimensions are set as $m = 2n$. We set
$q = -P x^\mathrm{cts}$, where $P = A^T A$,
and $x^\mathrm{cts}$ is randomly drawn
from the box $[0, 1]^n$. The problem is then scaled so that the simple
lower bound is $-1$, \ie, $f^\mathrm{cts} = -q^T P^\dagger q = -1$.

There are other ways to generate random problem instances. For example, the
eigenspectrum of $P$ is controlled by the magnitude of $m$ relative to $n$. We
note that $P$ becomes a near-diagonal matrix as $m$ diverges to infinity,
because the columns of $A$ are uncorrelated. This makes the integer least
squares problem easier to solve. On the contrary, smaller $m$ makes the
problem harder to solve. Another way of generating random problem instances is
to construct $P$ from a predetermined eigenspectrum $\omega_1, \ldots,
\omega_n$, as $P = Q\diag(\omega) Q^T$, where $Q$ is a random rotation matrix.
This makes it easy to generate a matrix with a desired condition number. Our
method showed the same qualitative behavior on data generated in these
different ways, for larger or smaller $m$, and also for different
eigenspectra.

The SDP~(\ref{primal_relaxation}) was solved using CVX~\cite{cvx,gb08} with
the MOSEK 7.1 solver~\cite{mosek}, on a 3.40 GHz Intel Xeon machine. For
problems of relatively small size $n \le 70$, we found the optimal point
using MILES~\cite{chang2007miles}, a branch-and-bound algorithm for mixed-integer
least squares problems, implemented in MATLAB. MILES
solves~(\ref{problem_statement}) by enumerating lattice points in a suitably
chosen ellipsoid. The enumeration method is based on various algorithms
developed in~\cite{chang2005mlambda,agrell2002closest,schnorr1994lattice,fincke1985improved}.

\subsection{Results}\label{s-results}

\paragraph{Lower bounds.}
We compare various lower bounds on $f^\star$. In~\cite{buchheim2015ellipsoid},
three lower bounds on $f^\star$ are shown, which we denote by
$f^\mathrm{axp}$, $f^\mathrm{qax}$, and $f^\mathrm{qrd}$, respectively. These
bounds are constructed from underestimators of $f$ that have a \emph{strong
rounding property}, \ie, the integer point minimizing the function is obtained
by rounding the continuous solution. We denote the lower bound obtained by
solving the following trust region problem in~\cite{bienstock2010eigenvalue}
by $f^\mathrm{tr}$:
\[
\begin{array}{ll}
\mbox{minimize} & x^T P x + 2q^T x \\
\mbox{subject to} & \|x - x^\mathrm{cts}\|_2^2 \ge \|x^\mathrm{cts} - x^\mathrm{rnd}\|_2^2.
\end{array}
\]

To the best of our knowledge, there is no standard benchmark test set for
the integer least squares problem.
Thus, we compared the lower bounds on randomly generated problem instances; for
each problem size, we generated $100$ random problem instances. Note that in
all instances, the simple lower bound was $f^\mathrm{cts} = -1$. We found not
only that our method found a tighter lower bound on average, but also that our
method performed consistently better, \ie, in all problem instances, the SDP
based lower bound was higher than any other lower bound. We found that the
pairs of lower bounds $(f^\mathrm{axp}, f^\mathrm{qax})$ and
$(f^\mathrm{qrd}, f^\mathrm{tr})$ were practically equal, although
the results of~\cite{buchheim2015ellipsoid} show that they can be all different.
We conjecture that this disparity comes from different problem sizes and
eigenspectra of the random problem instances.
The solution $f^\star$ was not computed for $n > 70$ as MILES was unable
to find it within a reasonable amount of time.

\begin{table}
\begin{center}
\begin{tabular}{|c||c|c|c|c|c|c|}
\hline
$n$ & $f^\star$ & $f^\mathrm{sdp}$ & $f^\mathrm{axp}$ & $f^\mathrm{qax}$ & $f^\mathrm{qrd}$ & $f^\mathrm{tr}$ \\ \hline\hline
50   & $-0.8357$ & $-0.9162$ & $-0.9434$ & $-0.9434$ & $-0.9736$ & $-0.9736$ \\ \hline
60   & $-0.8421$ & $-0.9202$ & $-0.9459$ & $-0.9459$ & $-0.9740$ & $-0.9740$ \\ \hline
70   & $-0.8415$ & $-0.9212$ & $-0.9471$ & $-0.9471$ & $-0.9747$ & $-0.9747$ \\ \hline
100  &   N/A     & $-0.9268$ & $-0.9509$ & $-0.9509$ & $-0.9755$ & $-0.9755$ \\ \hline
500  &   N/A     & $-0.9401$ & $-0.9606$ & $-0.9606$ & $-0.9777$ & $-0.9777$ \\ \hline
1000 &   N/A     & $-0.9435$ & $-0.9630$ & $-0.9630$ & $-0.9781$ & $-0.9781$ \\ \hline
\end{tabular}
\end{center}
\caption{Average lower bound by number of variables.}
\label{lb_comparison}
\end{table}

To see how tight $f^\mathrm{sdp}$ compared to the simple lower bound is, we
focus on the set of $100$ random problem instances of size $n = 60$, and
show, in Figure~\ref{lb_gaps_small}, the distribution of the gap between
$f^\star$ and the lower bounds.

\begin{figure}
\begin{center}
\includegraphics[width=.6\textwidth]{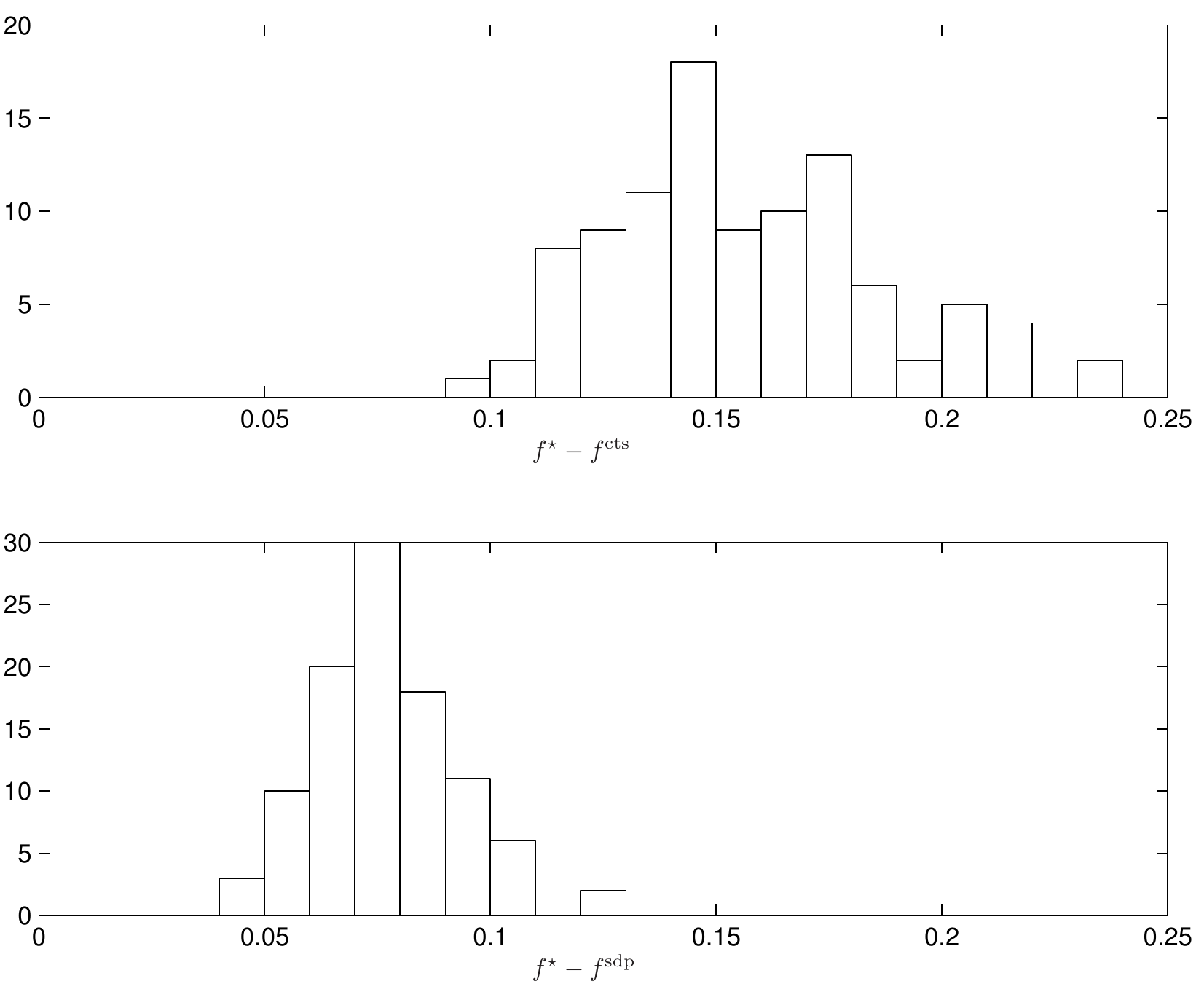}
\end{center}
\caption{Histograms of the gap between the optimal value $f^\star$ and the two lower bounds $f^\mathrm{cts}$ and $f^\mathrm{sdp}$, for $100$ random problem instances of size $n=60$.}
\label{lb_gaps_small} \end{figure}

\paragraph{Upper bounds.}
Algorithm~\ref{randalg} gives a better suboptimal solution as the number of
samples $K$ grows.
To test the
relationship between the number of samples and the quality of suboptimal
solutions, we considered a specific problem instance of size $n = 500$
and sampled $K = 50n$ points.
The result suggested that $K = 3n$ is a large enough number of samples for
most problems; in order to decrease ${\hat f}^\mathrm{best}$ further, many
more samples were necessary. All subsequent experiments discussed below used
$K=3n$ as the number of sample points.

In Table~\ref{ub_comparison}, we compare different upper bounds on $f^\star$
using the same set of test data considered above. We found
that Algorithm~\ref{randalg} combined with the 1-opt heuristic gives a
feasible point whose objective value is, on average, within $5\times 10^{-3}$
from the optimal value. The last column of Table~\ref{ub_comparison} indicates
the percentage of the problem instances
for which ${\hat f}^\mathrm{best} = f^\star$ held; we not only found near-optimal solutions, but for most
problems, the randomized algorithm actually terminated with the global
solution. We expect the same for larger problems, but have no evidence since
there is no efficient way to verify optimality.

\begin{table}
\begin{center}
\begin{tabular}{|c||c|c|c|c|c|c|}
\hline
$n$ & $f^\star$ & ${\hat f}^\mathrm{best}$ & $f^\mathrm{best}$ & ${\hat f}^\mathrm{rnd}$ & $f^\mathrm{rnd}$ & Optimal \\ \hline\hline
50   & $-0.8357$ & $-0.8353$ & $-0.8240$ & $-0.8186$ & $-0.7365$ & $90\%$ \\ \hline
60   & $-0.8421$ & $-0.8420$ & $-0.8268$ & $-0.8221$ & $-0.7397$ & $94\%$ \\ \hline
70   & $-0.8415$ & $-0.8412$ & $-0.8240$ & $-0.8235$ & $-0.7408$ & $89\%$ \\ \hline
100  &   N/A     & $-0.8465$ & $-0.8235$ & $-0.8296$ & $-0.7488$ & N/A    \\ \hline
500  &   N/A     & $-0.8456$ & $-0.7991$ & $-0.8341$ & $-0.7466$ & N/A    \\ \hline
1000 &   N/A     & $-0.8445$ & $-0.7924$ & $-0.8379$ & $-0.7510$ & N/A    \\ \hline
\end{tabular}
\end{center}
\caption{Average upper bound by number of variables.}
\label{ub_comparison}
\end{table}

We take the same test data used to produce Figure~\ref{lb_gaps_small}, and
show histograms of the suboptimality of $x^\mathrm{rnd}$, ${\hat
x}^\mathrm{rnd}$, $x^\mathrm{best}$, and ${\hat x}^\mathrm{best}$. The mean
suboptimality of $x^\mathrm{rnd}$ was $0.1025$, and simply finding a 1-opt
point from $x^\mathrm{rnd}$ improved the mean suboptimality to $0.0200$.
Algorithm~\ref{randalg} itself, without 1-opt refinement, produced suboptimal
points of mean suboptimality $0.0153$, and running Algorithm~\ref{oneoptalg}
on top of it reduced the suboptimality to $0.0002$.

\begin{figure}
\begin{center}
\includegraphics[width=.6\textwidth]{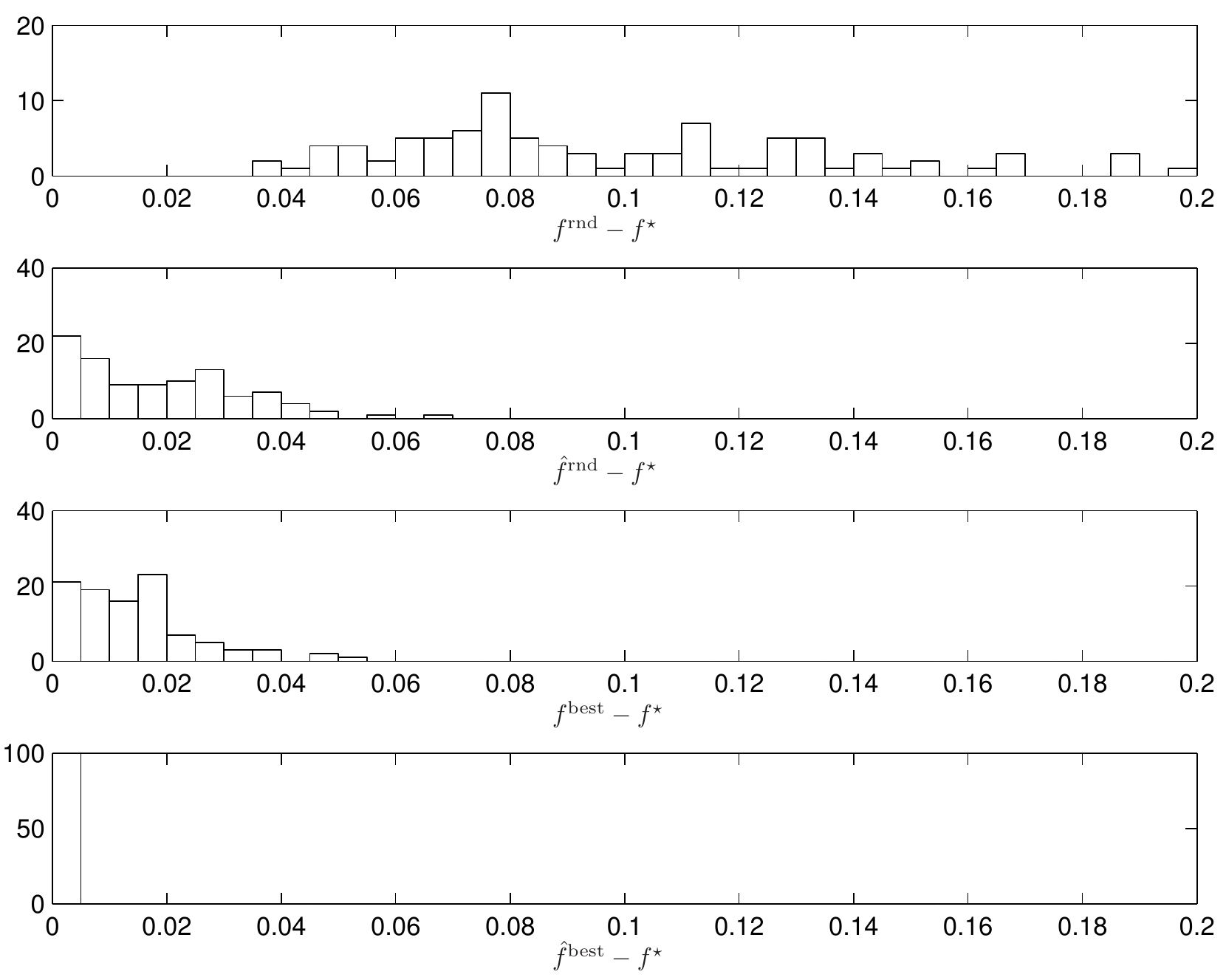}
\end{center}
\caption{Histograms of the suboptimality of $f^\mathrm{rnd}$, ${\hat f}^\mathrm{rnd}$, $f^\mathrm{best}$, and ${\hat f}^\mathrm{best}$, for $100$ random problem instances of size $n=60$.}
\label{ub_gaps_small} \end{figure}

Finally, we take problems of size $n = 1000$, where all existing
global methods run too slowly. As the optimal value
is unobtainable, we consider the gap given by the difference between the upper
and lower bounds. Figure~\ref{gaps_large} shows histograms of the four
optimality gaps obtained from our method, namely $f^\mathrm{rnd} -
f^\mathrm{cts}$, ${\hat f}^\mathrm{rnd} - f^\mathrm{cts}$, $f^\mathrm{best} -
f^\mathrm{sdp}$, and ${\hat f}^\mathrm{best} - f^\mathrm{sdp}$. The mean value
of these quantities were: $0.2490$, $0.1621$, $0.1511$, and $0.0989$. As seen
in Table~\ref{ub_comparison}, we believe that the best upper
bound ${\hat f}^\mathrm{best}$ is very close to the optimal value, whereas the
lower bound $f^\mathrm{sdp}$ is farther away from the optimal value.

\begin{figure}
\begin{center}
\includegraphics[width=.6\textwidth]{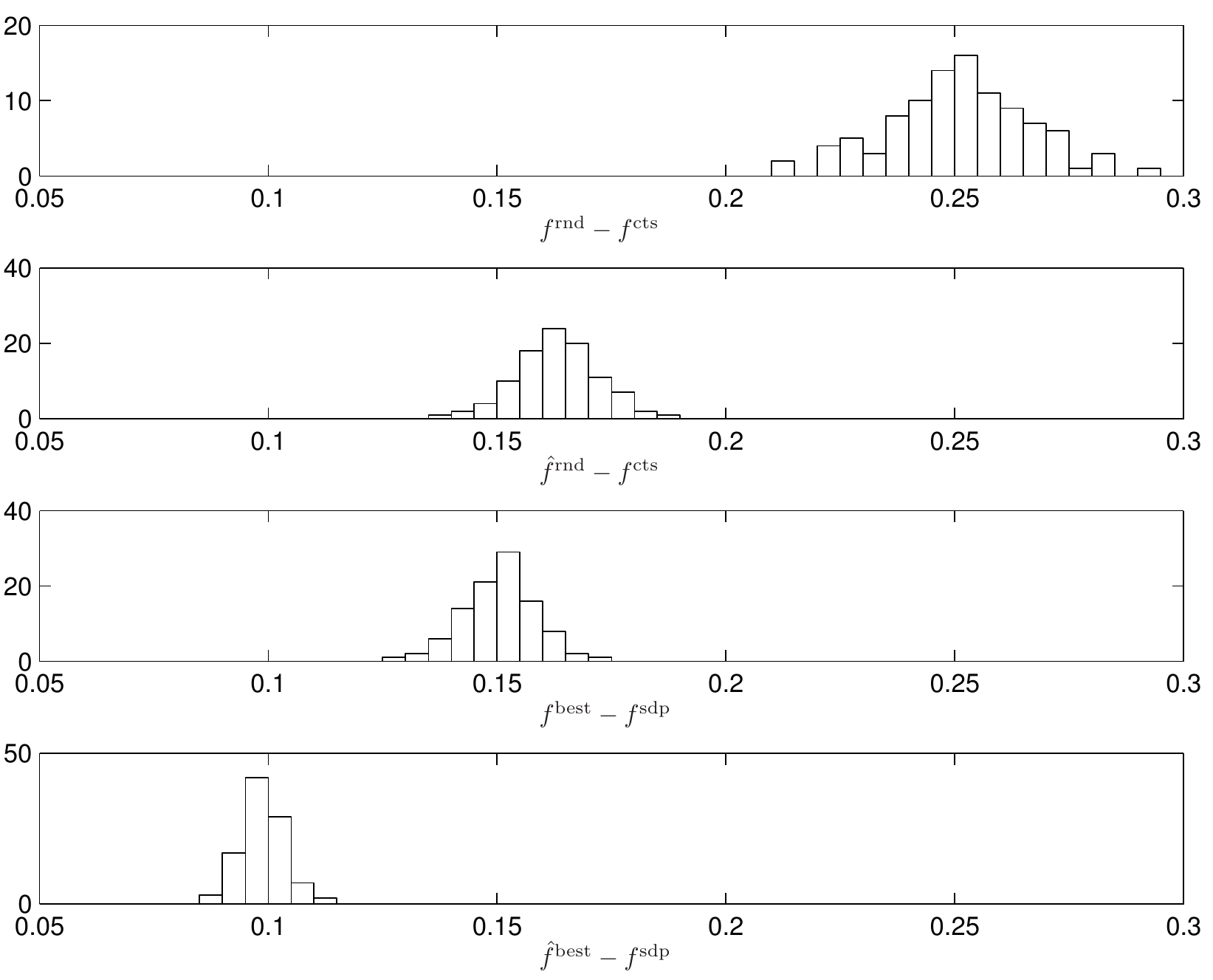}
\end{center}
\caption{Histograms of the four optimality gaps for $100$ random problem instances of large size $n=1000$.}
\label{gaps_large} \end{figure}

\paragraph{Running time.}
In Table~\ref{t-timing}, we compare the running time of our method and that of
MILES for problems of various sizes. The running time of MILES varied
depending on particular problem instances. For example, for $n = 70$, the
running time varied from $6.6$ seconds to $25$ minutes. Our method showed more
consistent running time, and terminated within 3 minutes for every problem
instance of the biggest size $n=1000$. It should be noted that MILES always
terminates with a global optimum, whereas our method does not have such a
guarantee, even though the experimental results suggest that the best suboptimal
point found is close to optimal.
From the breakdown of the running time of our method, we see that none of the
three parts of our method clearly dominates the total running time. We also
note that the total running time grows subcubically, despite the theoretical
running time of $O(n^3)$.

\begin{table}
\begin{center}
\begin{tabular}{|c||c|c||c|c|c|}
\hline
\multirow{2}{*}{$n$}
& \multicolumn{2}{c||}{Total running time}
& \multicolumn{3}{c|}{Breakdown of running time} \\ \cline{2-6}
& MILES & Our method & SDP & Random Sampling & Greedy 1-opt \\ \hline\hline
50   & $3.069$ & $0.397$ & $0.296$ & $0.065$ & $0.036$ \\ \hline
60   & $28.71$ & $0.336$ & $0.201$ & $0.084$ & $0.051$ \\ \hline
70   & $378.2$ & $0.402$ & $0.249$ & $0.094$ & $0.058$ \\ \hline
100  & N/A     & $0.690$ & $0.380$ & $0.193$ & $0.117$ \\ \hline
500  & N/A     & $20.99$ & $12.24$ & $4.709$ & $4.045$ \\ \hline
1000 & N/A     & $135.1$ & $82.38$ & $28.64$ & $24.07$ \\ \hline
\end{tabular}
\end{center}
\caption{Average running time of MILES and our method in seconds (left),
and breakdown of the running time of our method (right).}
\label{t-timing}
\end{table}

In a practical setting, if a near-optimal solution is good enough, then
running Algorithm~\ref{randalg} is more effective than enumeration algorithms.
Algorithm~\ref{randalg} is particularly useful if the problem size is $60$ or
more; this is the range where branch-and-bound type algorithms become unviable
due to their exponential running time. In practice, one may run an enumeration
algorithm (such as MILES) and terminate after a certain amount of time and use
the best found point as an approximation.
To show that Algorithm~\ref{randalg}
is a better approach of obtaining a good suboptimal solution, we compare our
method against MILES in the following way. First, we compute ${\hat
f}^\mathrm{best}$ via Algorithm~\ref{randalg}, and record the running time
$T$. Then, we run MILES on the same problem instance, but with time limit $T$,
\ie, we terminate MILES when its running time exceeds $T$, and record the best
suboptimal point found. Let $f^\mathrm{miles}$ be the objective value of this
suboptimal point. Table~\ref{t-ub-comparison} shows the average value of
$f^\mathrm{miles} - {\hat f}^\mathrm{best}$ and the percentage of problem
instances for which ${\hat f}^\mathrm{best} \le f^\mathrm{miles}$ held. The
experiment was performed on $100$ random problem instances of each size. We
observe that on every problem instance of size $n \ge 100$, our method produced
a better point than MILES, when alloted the same running time.

\begin{table}
\begin{center}
\begin{tabular}{|c||c|c|}
\hline
$n$ & $f^\mathrm{miles} - {\hat f}^\mathrm{best}$ & ${\hat f}^\mathrm{best} \le f^\mathrm{miles}$ \\ \hline\hline
50   & $0.0117$ & $98\%$ \\ \hline
60   & $0.0179$ & $99\%$ \\ \hline
70   & $0.0230$ & $99\%$ \\ \hline
100  & $0.0251$ & $100\%$ \\ \hline
500  & $0.0330$ & $100\%$ \\ \hline
1000 & $0.0307$ & $100\%$ \\ \hline
\end{tabular}
\end{center}
\caption{Comparison of the best suboptimal point found by our method and by MILES, when alloted the same running time.}
\label{t-ub-comparison}
\end{table}

There is no simple bound for the number of iterations that
Algorithm~\ref{oneoptalg} takes, as it depends on both $P$ and $q$, as well as
the initial point. In Table~\ref{t-oneoptiters}, we give the average number of
iterations for Algorithm~\ref{oneoptalg} to terminate at a 1-opt point, when
the initial points are sampled from the distribution found in~\S\ref{s-rand}.
We see that the number of iterations when $n \le 1000$ is roughly $0.14n$.
Although the asymptotic growth in the number of iterations appears to be
slightly superlinear, as far as practical applications are concerned, the
number of iterations is effectively linear. This is because the
SDP~(\ref{primal_relaxation}) cannot be solved when the problem becomes much
larger (\eg, $n > 10^5$).

\begin{table}
\begin{center}
\begin{tabular}{|c||c|}
\hline
$n$ & Iterations \\ \hline\hline
$50$   & $6.06$ \\ \hline
$60$   & $7.35$ \\ \hline
$70$   & $8.59$ \\ \hline
$100$  & $11.93$ \\ \hline
$500$  & $65.89$ \\ \hline
$1000$ & $141.1$ \\ \hline
\end{tabular}
\end{center}
\caption{Average number of iterations of Algorithm~\ref{oneoptalg}.}
\label{t-oneoptiters}
\end{table}

\paragraph{Branch-and-bound method.}
The results of Table~\ref{t-ub-comparison} suggest that enumeration methods
that solve~(\ref{problem_statement}) globally can utilize
Algorithm~\ref{randalg} by taking the suboptimal solution ${\hat
x}^\mathrm{best}$ and using its objective value ${\hat f}^\mathrm{best}$ as
the initial bound on the optimal value.
In Table~\ref{t-bnb-speedup}, we show
the average running time of MILES versus $n$, when different initial bounds
were provided to the algorithm: $0$, $f^\mathrm{rnd}$, ${\hat
f}^\mathrm{rnd}$, ${\hat f}^\mathrm{best}$, and $f^\star$. For a fair
comparison, we included the running time of computing the respective upper
bounds in the total execution time, except in the case of $f^\star$. We found
that when $n = 70$, if we start with ${\hat f}^\mathrm{best}$ as the initial
upper bound of MILES, the total running time until the global solution is
found is roughly $24\%$ lower than running MILES with the trivial upper
bound of $0$. Even when $f^\star$ is provided as the initial bound,
branch-and-bound methods will still traverse a search tree to look for a
better point (though it will eventually fail to do so).
This running time, thus, can be thought as the baseline performance.
If we compare the running
time of the methods with respect to this baseline running time, the effect
of starting with a tight upper bound becomes more apparent. When $0$ is given
as the initial upper bound, MILES spends roughly $3$ more minutes traversing
the nodes that would have been pruned if it started with $f^\star$ as the
initial upper bound. However, if the initial bound is
${\hat f}^\mathrm{best}$, then the additional time spent
is only $10$ seconds, as opposed to $3$ minutes.
Finally, we note that the baseline running time accounts for more than $70\%$
of the total running time even when the trivial upper bound of zero is given;
this suggests that the core difficulty in the problem lies more in proving the
optimality than in finding an optimal point itself.
In our experiments, the simple lower bound $f^\mathrm{cts}$ was used to
prune the search tree. If tighter lower bounds are used instead, this
baseline running time changes, depending on how easy it is to evaluate
the lower bound, and how tight the lower bound is.

\begin{table}
\begin{center}
\begin{tabular}{|c||c|c|c|c|c|}
\hline
\multirow{2}{*}{$n$} & \multicolumn{5}{c|}{Initial upper bound} \\ \cline{2-6}
& $0$ & $f^\mathrm{rnd}$ & ${\hat f}^\mathrm{rnd}$ & ${\hat f}^\mathrm{best}$ & $f^\star$ \\ \hline\hline
50   & $3.046$ & $3.039$ & $2.941$ & $2.900$ & $2.537$ \\ \hline
60   & $29.02$ & $29.09$ & $28.14$ & $24.07$ & $23.56$ \\ \hline
70   & $379.6$ & $379.4$ & $361.1$ & $290.0$ & $280.6$ \\ \hline
\end{tabular}
\end{center}
\caption{Average running time of MILES, given different initial upper bounds.}
\label{t-bnb-speedup}
\end{table}

Directly using the SDP-based lower bound to improve the performance of
branch-and-bound methods is more challenging,
due to the overhead of solving the SDP.
The results by~\cite{buchheim2015ellipsoid}
suggest that in order for a branch-and-bound scheme to achieve faster running
time, quickly evaluating a lower bound is more important than the quality of
the lower bound itself. Even if our SDP-based lower bound is superior to any
other lower bound shown in related works, solving an SDP at every node in the
branch-and-bound search tree is computationally too costly. Indeed, the
SDP-based axis-parallel ellipsoidal bound in~\cite{buchheim2015ellipsoid} fails to
improve the overall running time of a branch-and-bound algorithm when applied
to every node in the search tree. An outstanding open problem is to find an
alternative to $f^\mathrm{sdp}$ that is quicker to compute. One possible
approach would be to look for (easy-to-find) feasible points
to~(\ref{dual_relaxation}); as noted in~\S\ref{s-lagrange}, it is not
necessary to solve (\ref{dual_relaxation}) optimally in order to compute a
lower bound, since any feasible point of it yields a lower bound on
$f^\star$. (See, \eg,~\cite{dong2016relaxing}.)

\subsection*{Acknowledgments}
We thank three anonymous referees for providing helpful comments and constructive remarks.

\nocite{*}\bibliography{refs}
\end{document}